\documentclass[preprint,11pt]{elsarticle}
\usepackage{amsmath,amssymb,amsthm}
\usepackage{enumitem}
\usepackage{url}
\usepackage[colorlinks=true,linkcolor=blue,citecolor=blue,urlcolor=blue]{hyperref}

\newtheorem{definition}{Definition}[section]
\newtheorem{lemma}[definition]{Lemma}
\newtheorem{proposition}[definition]{Proposition}
\newtheorem{theorem}[definition]{Theorem}
\newtheorem{corollary}[definition]{Corollary}

\newtheorem{remark}[definition]{Remark}

\newcommand{\op}{\oplus}
\newcommand{\od}{\odot}
\newcommand{\omin}{\ominus}

\newcommand{\cl}{\operatorname{cl}}

\newcommand{\bet}{\beta}
\newcommand{\down}{\mathord{\downarrow}}
\begin{document}

\begin{frontmatter}

\title{Pseudocompact Topological \(MV\)-Algebras}

\author[addr1]{Li-Hong Xie}
\author[addr2]{Jiang Yang\corref{cor1}}
\ead{yunli198282@126.com (L.H. Xie), yangjiangdy@126.com}
\cortext[cor1]{Corresponding author.}

\address[addr1]{School of Mathematics and Computational Science, Wuyi University, Jiangmen, Guangdong, 529000, P.R. China}
\address[addr2]{School of Mathematical Sciences, Guangxi Minzu University, Nanning, 530006, P.R. China}

\begin{abstract}
Recently, topological MV-algebras have been investigated by several mathematicians. In this paper, we find that every topological \(MV\)-algebra is a Mal'tsev space introduced by Mal'tsev in 1954.  Hence, applying the theorem of Reznichenko and Uspenskij on pseudocompact Mal'tsev spaces, we show that the product of arbitrary family of pseudocompact topological \(MV\)-algebras are pseudocompact.  We also prove that every $\sigma$-compact topological \(MV\)-algebra is ccc. Secondly, we obtain that the Stone-\v{C}ech compactification of a pseudocompact topological \(MV\)-algebra carries a natural compact topological \(MV\)-algebra structure extending the original one. Finally, we prove that: let \(I\) be a closed ideal in a pseudocompact topological \(MV\)-algebra \(A\) and \(\iota_1:A\hookrightarrow\bet A\) is the naturally injective; then \(\cl_{\beta A}\iota_1(I)\) is a closed ideal of \(\beta A\) and \( \beta A/\cl_{\beta A}\iota_1(I)\cong \beta(A/I)\).
\end{abstract}


\end{frontmatter}

\section{Introduction}

Throughout the paper all spaces are assumed to be Tychonoff; unless otherwise explicitly stated.

 To investigate many-valued logic by algebraic means, MV-algebras was introduced by Chang \cite{Chang1958} in order to show {\L}ukasiewicz logic to be standard complete,
which are among the most important structures associated with logical systems, where``MV'' is short for ``many-value''. Similar to topological groups, Hoo \cite{Hoo1997} introduced the notion of a topological MV-algebra, which means an MV-algebra $(A,\oplus,\ast,0)$ with a topology such that the operations
$\oplus$ and $\ast$ are continuous functions. Some fundamental properties are investigated by Hoo. Recently, topological MV-algebras have been investigated by several mathematicians.

Pseudocompactness is one of the classical compactness-like properties in general topology. A Tychonoff space is \emph{pseudocompact} if every continuous real-valued function on it is bounded.  Unlike compactness, pseudocompactness is not stable under arbitrary products in general topology.  A fundamental theorem of Comfort and Ross states that:

\begin{theorem}\cite[Theorem 1.4]{ComfortRoss1966}.
 The product of any family of pseudocompact topological groups is pseudocompact.
\end{theorem}

This result is one of the best-known examples of the fact that topological groups behave much better than arbitrary topological spaces with respect to products. In many treatments of topological groups this phenomenon is naturally tied to canonical uniformities and completions, especially the two-sided or Ra\v{\i}kov completion; see, for instance, Arhangel'skii and Tkachenko \cite{ArhangelskiiTkachenko2008}.

A space $X$ is ccc. or has the {\it Suslin property}, if every disjoint family of nonempty open sets in $X$ is at most countable. A space is {\it $\sigma$-compact} if it is the union of countably
many compact sets.  Another classical theorem in topological groups obtained by Tkachenko:

\begin{theorem}\cite{Tkachenko1983}
Every $\sigma$-compact topological group is ccc.
\end{theorem}

The first purpose of the present paper is to verify the corresponding phenomenon for topological \(MV\)-algebras.  Since \(MV\)-algebras are neither groups nor rings, one should not expect the proof for topological groups to transfer directly.  Indeed, a topological \(MV\)-algebra has no invertible translations.  For example, in the standard \(MV\)-algebra \([0,1]\) with
\[
        x\op y=\min\{1,x+y\},
\]
the map \(y\mapsto x\op y\) is usually neither one-to-one nor onto when \(x>0\).  Thus the topological group method based on translations and Ra\v{\i}kov completion is not the right explanation.

The second purpose of the present paper is to study the following questions raised by the Stone-\v Cech compactification of pseudocompact topological \(MV\)-algebras. How do ideals behave under \(A\mapsto\beta A\)?

The paper is organized as follows. Section \ref{sec:preliminaries} recalls the required notions. Section \ref{sec:malcev} proves that every topological \(MV\)-algebra is a Mal'tsev space. Applying this result, we show that the main product theorem in pseudocompact topological $MV$-algebra. In Section \ref{sec:beta}, we also show that the pseudocompact product theorem has a useful compactification consequence. If \(A\) is a Tychonoff pseudocompact topological \(MV\)-algebra, then the Stone-\v{C}ech compactification \(\beta A\) admits a natural compact topological \(MV\)-algebra structure extending the operations of \(A\). The proof combines the product theorem, Glicksberg's theorem on Stone-\v{C}ech compactifications of products \cite{Glicksberg1959}, and the density of \(A\) in \(\beta A\).  This is the \(MV\)-algebraic analogue of the familiar fact that pseudocompact topological groups have compact group compactifications.

\section{Preliminaries}\label{sec:preliminaries}

We recall the notation used throughout the paper.  An \(MV\)-algebra is an algebra
\[
        (A,\op,{}^*,0)
\]
of type \((2,1,0)\) satisfying the usual Chang axioms.  We put
\[
        1=0^*,\qquad
        x\od y=(x^*\op y^*)^*,\qquad
        x\omin y=x\od y^*.
\]
The natural order on \(A\) is defined by
\[
        x\leq y\quad\Longleftrightarrow\quad x^*\op y=1.
\]
With respect to this order, \(A\) is a bounded distributive lattice.  We use
\[
        x\vee y=y\op(x\omin y),
        \qquad
        x\wedge y=x\od(x^*\op y).
\]

For background on \(MV\)-algebras we refer to \cite{Cignoli2000,Mundici1986}.

We shall use the following standard facts.  For all \(x,y,z\in A\),
\begin{align}
        x&=(x\wedge y)\oplus(x\ominus y), \label{eq:decomp}\\
        x\ominus z&\leq (x\ominus y)\oplus(y\ominus z), \label{eq:triangle}\\
        x\ominus(x\ominus y)&=x\wedge y, \label{eq:double-minus}
\end{align}
for all \(x,y,z\in A\).  These are standard facts in the theory of \(MV\)-algebras; see \cite{Chang1958,Cignoli2000,Mundici1986}.
For \(n\geq 1\), write \(nx\) for the \(n\)-fold truncated sum \(x\oplus\cdots\oplus x\).

A subset \(I\subseteq A\) is an \emph{ideal} if \(0\in I\), if \(y\leq x\in I\) implies \(y\in I\), and if \(x,y\in I\) implies \(x\oplus y\in I\). The Chang distance on an \(MV\)-algebra \(A\) \(d:A\times A\rightarrow A\) is defined as:
\[
        d(x,y)=(x\ominus y)\oplus(y\ominus x),
\]
for each \(x,y\in A\).

If \(I\) is an ideal, the associated congruence is
\[
        x\equiv_I y \quad\Longleftrightarrow\quad d(x,y)\in I.
\]
Since ideals are downward closed, this is equivalent to
\[
        x\ominus y\in I
        \quad\text{and}\quad
        y\ominus x\in I.
\]
The quotient algebra is denoted by \(A/I\).

A \emph{topological \(MV\)-algebra} is an \(MV\)-algebra \(A\) equipped with a topology such that
\[
        \op:A\times A\longrightarrow A,
        \qquad
        {}^*:A\longrightarrow A
\]
are continuous. Since \(\od,\omin,\wedge,\vee\) are term operations, they are continuous in every topological \(MV\)-algebra.  Hoo introduced and studied topological \(MV\)-algebras in \cite{Hoo1997}.  Weber investigated locally convex topological \(MV\)-algebras and their relationship with uniform \(MV\)-algebras in \cite{Weber2012}; more recently, Li and Yang developed the uniformization theory of locally convex topological \(MV\)-algebras and gave examples showing that topological \(MV\)-algebras need not be normal \cite{LiYang2026}.

 The main product theorem of pseudocompact topological \(MV\)-algebras  is formulated for the same setting in which the Reznichenko--Uspenskij theorem on pseudocompact Mal'tsev spaces.

\begin{theorem}\cite[Reznichenko--Uspenskij]{ReznichenkoUspenskij1998}\label{thm:RU}
The product of any family of pseudocompact Mal'tsev spaces is pseudocompact.
\end{theorem}
\begin{definition}\label{def:malcev-space}\cite{Mal'tsev1954}
A topological space \(X\) is called a \emph{Mal'tsev space} if there exists a continuous map
\[
        m:X^3\longrightarrow X
\]
such that
\[
        m(x,y,y)=m(y,y,x)=x
\]
for all \(x,y\in X\). Such a map is called a \emph{Mal'tsev operation} on \(X\).
\end{definition}

Every topological group is a Mal'tsev space: if \(G\) is a topological group, then
\[
        m(x,y,z)=xy^{-1}z
\]
is a continuous Mal'tsev operation.

 \begin{theorem}\cite[Uspenskij]{Uspenskij1985}\label{thm:U}
 Every $\sigma$-compact Mal'tsev space is ccc.
 \end{theorem}
 Theorems \ref{thm:RU} and \ref{thm:U} showed that the Comfort--Ross product theorem for pseudocompact topological groups and the Suslin property for $\sigma$-compact topological groups are special cases of a more general theorem for Mal'tsev spaces, respectivly.

If \(h : A \rightarrow B\) is a homomorphism between two \(MV\)-algebras \(A\) and \(B\), then we use the symbol \(\text{ker}(h)\) to
denote the kernel of \(h\), where \(\text{ker}(h) := \{x \in A | h (x) = 0_B \}\).

\begin{proposition}\cite[Corollary 3.2.]{YinXieYang2026}\label{Pro:YinXieYang}
 Suppose that \( A, B \) and \( C \) are topological MV-algebras, \( \varphi: A \to B \) and \( \psi: A \to C \) are continuous surjective homomorphisms such that \( \ker(\psi) \subseteq \ker(\varphi) \). If the homomorphism \( \psi \) is open, then there exists a continuous homomorphism \( f: C \to B \) such that \( \varphi = f \circ \psi \).
\end{proposition}

\section{The Mal'tsev mechanism in topological \(MV\)-algebras}\label{sec:malcev}

We now verify that every topological \(MV\)-algebra is an Mal'tsev space.  This is the key algebraic reason behind the pseudocompact product theorem.

\begin{proposition}\label{prop:mv-malcev}
Let \(A\) be an \(MV\)-algebra.  Define
\[
        m(x,y,z)=((x\omin y)\op z)\wedge((z\omin y)\op x).
\]
Then \(m\) is a Mal'tsev term; that is,
\[
        m(x,y,y)= m(y,y,x)=x
\]
for all \(x,y\in A\).  Consequently, the underlying space of every topological \(MV\)-algebra is a Mal'tsev space.
\end{proposition}

\begin{proof}
Since \(m\) is built from the \(MV\)-operations, it is an \(MV\)-term.  First,
\[
\begin{aligned}
        m(x,y,y)
        &=((x\omin y)\op y)\wedge((y\omin y)\op x)  \\
        &=((x\omin y)\op y)\wedge x.
\end{aligned}
\]
By \eqref{eq:decomp},
\[
        x=(x\wedge y)\op(x\omin y).
\]
Since \(x\wedge y\leq y\) and \(\op\) is order-preserving, we have
\[
        x=(x\wedge y)\op(x\omin y)
        \leq y\op(x\omin y)=(x\omin y)\op y.
\]
Therefore
\[
        ((x\omin y)\op y)\wedge x=x,
\]
and hence
\[
        m(x,y,y)=x.
\]
Similarly,
\[
\begin{aligned}
        m(y,y,x)
        &=((y\omin y)\op x)\wedge((x\omin y)\op y) \\
        &=x\wedge((x\omin y)\op y).
\end{aligned}
\]
The same inequality \(x\leq (x\omin y)\op y\) gives
\[
        m(y,y,x)=x.
\]
Thus \(m\) is a Mal'tsev term.  If \(A\) is a topological \(MV\)-algebra, all term operations are continuous.  Hence \(m:A^3\to A\) is continuous, and the underlying topological space of \(A\) is a Mal'tsev space.
\end{proof}

\begin{remark}\label{rem:UA}
The existence of a Mal'tsev term also explains why many Hausdorff-type separation properties of topological groups have analogues for topological \(MV\)-algebras.  In universal algebra, the variety of \(MV\)-algebras is congruence permutable.  Taylor's theorem, and later work of Kearnes and Sequeira, show that in congruence-permutable varieties a \(T_0\) topological algebra is Hausdorff; see \cite{Taylor1977,KearnesSequeira2002}.  The present paper uses a different consequence of the same Mal'tsev phenomenon: the product theorem for pseudocompact Mal'tsev spaces.
\end{remark}

We now prove the first main theorem.

\begin{theorem}\label{thm:pseudocompact-product}
Let \(\{A_\lambda:\lambda\in\Lambda\}\) be a family of pseudocompact topological \(MV\)-algebras whose underlying spaces are Tychonoff.  Then the product
\[
        A=\prod_{\lambda\in\Lambda} A_\lambda,
\]
with the product topology and pointwise \(MV\)-operations, is a pseudocompact topological \(MV\)-algebra.
\end{theorem}

\begin{proof}
First, \(A\) is a topological \(MV\)-algebra.  Indeed, the operations are defined coordinatewise:
\[
        (x_\lambda)_\lambda\op(y_\lambda)_\lambda
        =(x_\lambda\op y_\lambda)_\lambda,
        \qquad
        (x_\lambda)_\lambda^*=(x_\lambda^*)_\lambda.
\]
The continuity of these maps follows from the continuity of the coordinate operations on each \(A_\lambda\) and the definition of the product topology.

It remains to prove pseudocompactness.  By Proposition \ref{prop:mv-malcev}, each \(A_\lambda\) is a Mal'tsev space.  By hypothesis, each \(A_\lambda\) is pseudocompact. Hence each \(A_\lambda\) is a pseudocompact Mal'tsev space.  By the theorem of Reznichenko and Uspenskij, the product of any family of pseudocompact Mal'tsev spaces is pseudocompact.  Therefore \(A=\prod_{\lambda\in\Lambda} A_\lambda\) is pseudocompact.
\end{proof}

\begin{corollary}\label{cor:powers}
If \(A\) is a pseudocompact Tychonoff topological \(MV\)-algebra and \(\Lambda\) is any index set, then the power \(A^\Lambda\), with pointwise \(MV\)-operations, is a pseudocompact topological \(MV\)-algebra.
\end{corollary}

\begin{proof}
Apply Theorem \ref{thm:pseudocompact-product} to the constant family \(A_\lambda=A\).
\end{proof}

\begin{remark}
Theorem \ref{thm:pseudocompact-product} shows that topological \(MV\)-algebras inherit a substantial part of the Comfort--Ross phenomenon even though they are not groups.  This places topological \(MV\)-algebras within the broader class of topological Mal'tsev algebras.  Thus, the product theorem is not an accident of the particular \(MV\)-operations; it follows from a robust Mal'tsev mechanism.
\end{remark}

Applying Theorem \ref{thm:U} and Proposition \ref{prop:mv-malcev} we obtain the second main theorem:
\begin{theorem}
Every $\sigma$-compact topological \(MV\)-algebra is ccc.
\end{theorem}

\section{Stone--\v Cech compactifications as compact MV-algebras}\label{sec:beta}
We now derive a compactification consequence. Those results are analogous to the compact group compactification of a pseudocompact topological group, but its proof uses the Mal'tsev-space product theorem and Glicksberg's theorem on Stone-\v{C}ech compactifications of products.

\begin{theorem}\label{thm:beta-compactification}
Let \(A\) be a pseudocompact topological \(MV\)-algebra.  Then the Stone-\v{C}ech compactification \(\beta A\) carries a natural compact topological \(MV\)-algebra structure such that the canonical embedding
\[
        e_A:A\longrightarrow \beta A
\]
is a dense topological \(MV\)-embedding.
\end{theorem}

\begin{proof}
Since \(A\) is pseudocompact and is a Mal'tsev space by Proposition \ref{prop:mv-malcev}, the product \(A\times A\) is pseudocompact by Theorem \ref{thm:RU}.  By Glicksberg's theorem on Stone-\v{C}ech compactifications of products \cite{Glicksberg1959}, the canonical map
\[
        \beta(A\times A)\longrightarrow \beta A\times \beta A
\]
is a homeomorphism.

Consider the continuous map
\[
        e_A\circ \op:A\times A\longrightarrow \beta A,
\]
where \(\op:A\times A\to A\) is the \(MV\)-sum.  By the universal property of \(\beta(A\times A)\), this map extends uniquely to a continuous map
\[
        \widehat{\op}:\beta(A\times A)\longrightarrow \beta A.
\]
Using the identification \(\beta(A\times A)=\beta A\times \beta A\), we regard \(\widehat{\op}\) as a continuous binary operation
\[
        \widehat{\op}:\beta A\times \beta A\longrightarrow \beta A.
\]
Similarly, the continuous map
\[
        e_A\circ{}^*:A\longrightarrow \beta A
\]
extends uniquely to a continuous map
\[
        \widehat{*}:\beta A\longrightarrow \beta A.
\]
Put
\[
        \widehat 0=e_A(0).
\]
We claim that
\[
        (\beta A,\widehat{\op},\widehat{*},\widehat 0)
\]
is an \(MV\)-algebra.  Each \(MV\)-identity is an equality between two continuous term functions on a finite power \((\beta A)^n\). These two continuous functions agree on the dense subspace \(A^n\), because the identities hold in \(A\). Since \(\beta A\) is Hausdorff, the two term functions agree on all of \((\beta A)^n\).  Therefore all \(MV\)-identities hold on \(\beta A\).

The operations \(\widehat{\op}\) and \(\widehat{*}\) are continuous by construction, so \(\beta A\) is a compact topological \(MV\)-algebra. The embedding \(e_A\) preserves \(0\), \(\op\), and \({}^*\), again by construction. Its image is dense in \(\beta A\).  Thus \(e_A:A\to\beta A\) is a dense topological \(MV\)-embedding.
\end{proof}

\begin{corollary}\label{cor:beta-product}
Let \(\{A_\lambda:\lambda\in\Lambda\}\) be a family of pseudocompact topological \(MV\)-algebras. Then
\[
        \beta\left(\prod_{\lambda\in\Lambda}A_\lambda\right)
\]
has a natural compact topological \(MV\)-algebra structure.  Moreover, by Glicksberg's theorem,
\[
        \beta\left(\prod_{\lambda\in\Lambda}A_\lambda\right)
        \cong
        \prod_{\lambda\in\Lambda}\beta A_\lambda
\]
as compact spaces, and the isomorphism is compatible with the coordinatewise \(MV\)-operations.
\end{corollary}

\begin{proof}
The product \(\prod_{\lambda\in\Lambda}A_\lambda\) is pseudocompact by Theorem \ref{thm:pseudocompact-product}.  Hence Theorem \ref{thm:beta-compactification} applies. Glicksberg's theorem identifies the Stone-\v{C}ech compactification of the product with the product of the Stone-\v{C}ech compactifications whenever the product is pseudocompact.  Since all operations are defined coordinatewise and the extensions are unique, the resulting compactification isomorphism respects the \(MV\)-operations.
\end{proof}

\begin{remark}
Theorem \ref{thm:beta-compactification} gives an important compactification tool. If \(A\) is pseudocompact, then \(A\) embeds densely into a compact topological \(MV\)-algebra \(\beta A\). This is useful in two ways.  First, many questions about \(A\) can be studied inside the compact \(MV\)-algebra \(\beta A\).  Second, continuous homomorphisms from \(A\) into compact Hausdorff topological \(MV\)-algebras can often be extended through \(\beta A\), by the universal property of the Stone-\v{C}ech compactification.
\end{remark}

\begin{theorem}\label{thm:compact-reflection}
Let \(A\) be a pseudocompact topological \(MV\)-algebra and let \(K\) be a compact topological \(MV\)-algebra. Every continuous \(MV\)-homomorphism
\[
        f:A\to K
\]
extends uniquely to a continuous \(MV\)-homomorphism
\[
        \beta f:\beta A\to K.
\]
Consequently, \(\beta\) is the compact reflection on the category of Tychonoff pseudocompact topological \(MV\)-algebras and continuous \(MV\)-homomorphisms.
\end{theorem}

\begin{proof}
Since \(K\) is compact Hausdorff, the continuous map \(f:A\to K\) extends uniquely to a continuous map \(\beta f:\beta A\to K\).  It remains to check that \(\beta f\) is an \(MV\)-homomorphism.  For addition, the two maps
\[
        \beta A\times\beta A\to K,
        \qquad
        (u,v)\mapsto \beta f(u\widehat\oplus v)
\]
and
\[
        (u,v)\mapsto \beta f(u)\oplus \beta f(v)
\]
are continuous and agree on the dense subset \(A\times A\).  Hence they agree everywhere.  The proof for involution and for \(0\) is identical.  Thus \(\beta f\) is an \(MV\)-homomorphism.  The uniqueness is the usual uniqueness in the Stone--\v Cech extension property.
\end{proof}

\begin{corollary}\label{cor:beta-functor}
If \(h:A\to B\) is a continuous \(MV\)-homomorphism between Tychonoff pseudocompact topological \(MV\)-algebras, then the Stone--\v Cech extension
\[
        \beta h:\beta A\to\beta B
\]
is a continuous \(MV\)-homomorphism.  If \(h\) has dense image, then \(\beta h\) is onto.  In particular, if \(h\) is onto, then \(\beta h\) is onto.
\end{corollary}

\begin{proof}
Apply Theorem \ref{thm:compact-reflection} with \(K=\beta B\) and with the map \(e_B\circ h:A\to\beta B\).  If \(h(A)\) is dense in \(B\), then \(e_B(h(A))\) is dense in \(\beta B\).  The image \(\beta h(\beta A)\) is compact, hence closed, and contains this dense subset; therefore it is all of \(\beta B\).
\end{proof}

\begin{proposition}\label{prop:closure-ideal}
Let \(I\) be an ideal of a pseudocompact topological \(MV\)-algebra \(A\) and \(\iota_1:A\hookrightarrow\bet A\) the naturally injective. Then \(\cl_{\beta A}\iota_1(I)\) is a closed ideal of \(\beta A\).
\end{proposition}

\begin{proof}
Let \(J=\cl_{\beta A}\iota_1(I)\).  Clearly \(0\in J\).  If \((u,v)\in J\times J\), then there a net \((u_\alpha,v_\alpha)\in I\times I\) such that converging to \((u,v)\). Thus, the nets \((u_\alpha)\) and \((v_\alpha)\) converging to \(u\)and \(v\), respectively. Since \(\oplus\) is continuous, the net \(u_\alpha\oplus v_\alpha\), indexed by the product directed set, belongs to \(I\) and converges to \(u\oplus v\).  Hence \(u\oplus v\in J\).

It remains to prove downward closedness.  Suppose \(0\leq b\leq c\in J\).  Choose a net \((c_\alpha)\) in \(I\) converging to \(c\).  Since \(A\) is dense in \(\beta A\), choose a net \((b_\gamma)\) in \(A\) converging to \(b\).  Then
\[
        b_\gamma\wedge c_\alpha\longrightarrow b\wedge c=b.
\]
For each \((\gamma,\alpha)\), we have \(b_\gamma\wedge c_\alpha\leq c_\alpha\in I\).  Since \(I\) is downward closed, \(b_\gamma\wedge c_\alpha\in I\).  Hence \(b\in J\).  Therefore \(J\) is an ideal.
\end{proof}

\begin{proposition}\cite[Corollary 3.5]{XieYang}\label{prop:natural-quotients}
Let \(A\) be a topological \(MV\)-algebra and \(I\) an ideal of \(A\).  Then \(A/I\), endowed with the quotient topology, is a topological \(MV\)-algebra, and the natural quotient homomorphism
\[
        q:A\to A/I
\]
is an open map.
\end{proposition}

Let $A$ be \(MV\)-algebra, for a subset \(B\subseteq A\), define its downward hull by
\[
        \down B=\{a\in A:\text{ there exists }b\in B\text{ such that }a\leq b\}.
\]
\begin{definition}\cite{XieYang20262}\label{def:local-solid}
Let \(A\) be a topological \(MV\)-algebra.  We say that \(A\) is \emph{locally order-stable at zero}, or \emph{locally solid at zero}, if for every neighbourhood \(U\) of \(0\), there exists a neighbourhood \(V\) of \(0\) such that \(  \down V\subseteq U.\)
\end{definition}
For sake of completeness, we give out the proof.
\begin{proposition}\cite{XieYang20262}\label{Pro:XieYang2}
The following Hausdorff topological \(MV\)-algebras are locally order-stable at zero.
\begin{enumerate}
    \item[(a)] Topological \(MV\)-subalgebras of locally order-stable topological \(MV\)-algebras;
    \item [(b)] Continuous open homomorphic images of locally order-stable topological \(MV\)-algebras;
    \item [(c)] Compact topological \(MV\)-algebras.
\end{enumerate}
\end{proposition}

\begin{proof}
For (a), intersect a downward closed zero-neighbourhood base with the subalgebra.  For (b), if \(h:A\to B\) is continuous, open and onto, and if \(D\) is a downward closed zero-neighbourhood in \(A\), then \(h(D)\) is a zero-neighbourhood in \(B\).  It is downward closed: if \(b\leq h(d)\), choose \(a\) with \(h(a)=b\); then
\[
        b=b\wedge h(d)=h(a\wedge d),
\]
and \(a\wedge d\leq d\in D\), so \(a\wedge d\in D\).  Hence \(b\in h(D)\).

For (c), let \(A\) be compact and let \(U\) be a zero-neighbourhood.  Choose an open zero-neighbourhood \(U_0\subseteq U\).  The map \(m:A\times A\to A\), \(m(a,v)=a\wedge v\), is continuous and satisfies \(m(A\times\{0\})\subseteq U_0\).  By the tube lemma, there is a zero-neighbourhood \(V\) with \(A\times V\subseteq m^{-1}(U_0)\).  If \(0\leq b\leq v\in V\), then \(b=b\wedge v\in U_0\subseteq U\).  Thus \(\downarrow V\subseteq U\).
\end{proof}

\begin{lemma}\label{Lemma}
Every pseudocompact topological \(MV\)-algebra \(A\) is locally order-stable(equivalently locally convex).
\end{lemma}

\begin{proof}
Since \(A\) is a pseudocompact topological \(MV\)-algebra, by Theorem \ref{thm:beta-compactification} the Stone-\v{C}ech compactification \(\beta A\) is a compact topological \(MV\)-algebra and \(A\) as a topological \(MV\)-subalgebra of \(\beta A\). Hence, \(A\) is a locally order-stable topological \(MV\)-algebra by (a) and (c) in Proposition \ref{Pro:XieYang2}. The fact that both locally order-stable properties and locally convex properties are equivalentin topological \(MV\)-algebras (see \cite{XieYang20262}).
\end{proof}

\begin{theorem}\label{thm:quotient-exactness}
Let \(I\) be a closed ideal in a pseudocompact topological \(MV\)-algebra \(A\) and \(\iota_1:A\hookrightarrow\bet A\) is the naturally injective. Then \(\cl_{\beta A}\iota_1(I)\) is a closed ideal of \(\beta A\) and \( \beta A/\cl_{\beta A}\iota_1(I)\cong \beta(A/I)\).
\end{theorem}

\begin{proof}
By Proposition \ref{prop:natural-quotients} we obtain that the quotient \(MV\)-algebra \( A/I\), endowed with the quotient topology, is a topological \(MV\)-algebra. Firstly, we shall prove that the quotient topological \(MV\)-algebra \( A/I\) is Tychonoff. From Lemma \ref{Lemma} it follows that \(A\) is locally convex. Since \(I\) is closed in \(A\), the quotient space is \(T_1\) by the fact that in a topological \(MV\)-algebra \(B\), \(\{b\}\) closed for all \(b\in B\) if \(\{0\}\) is closed \cite[Proposition 3.3]{Hoo1997}. Thus, it follows that the space \( A/I\) is regular from the fact every \(T_1\) topological \(MV\)-algebra is regular \cite[Corollary 3.10]{GanLuandY2024}. Since every locally convex
regular topological \(MV\)-algebra is Tychonoff \cite[Corollary 4.16]{LiYang2026}, it is enough to show that \(A/I\) is locally convex. This follows that \(A/I\) is an open homomorphic image of \(A\) by (b) of Proposition \ref{Pro:XieYang2} and Proposition \ref{prop:natural-quotients}.

Let \(q:A\to A/I\) be the natural quotient homomorphism and \(\iota_2:A/I\hookrightarrow\bet (A/I)\) the naturally injective.

Because \(A\) is pseudocompact and \(q\) is continuous, \(A/I\) is also pseudocompact.
Hence, \(\bet(A/I)\) is a compact topological \(MV\)-algebra by Theorem \ref{thm:beta-compactification}.
By Theorem \ref{thm:compact-reflection}, the functoriality of the Stone-\v{C}ech compactification gives a continuous \(MV\)-homomorphism \[\bet(\iota_2\circ q):\bet A\rightarrow \bet(A/I)\] extending \(\iota_2\circ q\) such that  \[\bet(\iota_2\circ q)\circ \iota_1=\iota_2\circ q.\]
\(\bet(\iota_2\circ q)\) is surjective because the density of \(\iota_2(A/I)\) in \(\bet(A/I)\) and the fact that \(\bet(\iota_2\circ q)(\bet A)\) is compact and contains \(\iota_2(A/I)\).

Put \(K_q = (\bet(\iota_2\circ q))^{-1}(\iota_2(0/I))\). Then \(K_q\) is a closed ideal of \(\bet A\), because \(\iota_2(0/I)\) is closed in  \(\bet (A/I)\).

\(\cl_{\bet A} \iota_1(I)\) is an ideal in \(\bet A\) by Proposition \ref{prop:closure-ideal}. We first shall prove  \(\cl_{\bet A} \iota_1(I) \subseteq K_q\).
Since \(q(I)=\{0/I\}\) and \(\bet(\iota_2\circ q)\circ \iota_1=\iota_2\circ q\), we have \[\bet(\iota_2\circ q)(\iota_1(I))=\iota_2(q(I))=\{\iota_2(0/I)\}.\] The continuity of \(\bet(\iota_2\circ q)\) implies \[\bet(\iota_2\circ q)(\cl_{\bet A} \iota_1(I)) \subseteq \cl_{\bet(A/I)}\{\bet(\iota_2\circ q)(\iota_1(I))\} = \{\iota_2(0/I)\};\] therefore \(\cl_{\bet A} \iota_1(I) \subseteq K_q\).

Next, we shall prove the reverse inclusion \(K_q \subseteq \cl_{\bet A} \iota_1(I)\).
Consider the quotient topological \(MV\)-algebra \(\bet A / \cl_{\bet A} \iota_1(I)\). Let
\[\pi:\bet A\rightarrow \bet A / \cl_{\bet A} \iota_1(I)\]
be the natural quotient homomorphism. Because \(\cl_{\bet A} \iota_1(I)\) is a closed ideal, the quotient \(\bet A /\cl_{\bet A} \iota_1(I)\) is a compact Hausdorff topological \(MV\)-algebra. The composition \(\pi\circ \iota_1:\)
\[
A  {\overset{\iota_1}\hookrightarrow }\bet A  {\overset{\pi}\rightarrow} \bet A / \cl_{\bet A} \iota_1(I)
\]
is a continuous \(MV\)-homomorphism whose kernel is \[\iota_1^{-1}(\iota_1(A) \cap \cl_{\bet A}\iota_1(I)) = \cl_A I.\] But \(I\) is closed in \(A\) (since \(q\) is continuous and \(\{0/I\}\) is closed in \(A/I\)); hence \(\cl_A I = I\). Consequently we obtain a continuous injective \(MV\)-homomorphism
\[
\iota : A/I \longrightarrow \bet A / \cl_{\bet A} I
\]
such that \(\iota\circ q=\pi\circ \iota_1.\)

By Theorem \ref{thm:compact-reflection}, there exists a unique continuous \(MV\)-homomorphism
\[
\bet \iota : \bet(A/I) \longrightarrow \bet A / \cl_{\bet A} I
\]
such that \(\bet \iota\circ \iota_2  = \iota\). On the other hand, Note that \(\bet (\iota_2\circ q)\) and \(\pi\) are quotient homomorphisms, so from Proposition \ref{prop:natural-quotients} it follows that \(\bet (\iota_2\circ q)\) and \(\pi\) are open continuous homomorphism. Since \(\cl_{\bet A} \iota_1(I)\) and \(K_q\) are the kernels of \(\pi\) and \(\bet (\iota_2\circ q)\), respectively, and \(\cl_{\bet A} \iota_1(I) \subseteq K_q\), by Proposition \ref{Pro:YinXieYang}, there is a continuous homomorphism:
\[
\Psi : \bet A / \cl_{\bet A} \iota_1(I) \longrightarrow \bet(A/I)
\]
such that \(\Psi\circ \pi=\bet (\iota_2\circ q)\). Now it follows that \(\Psi\circ \iota=\iota_2\) from the fact that, for each \(a/I\in A/I\), the following holds:
\begin{align*}
\iota_2(a/I) &= \iota_2(q(a)) \quad\quad \quad\quad \quad\quad\text{by~} q(a)=a/I
\\&= \bet (\iota_2\circ q)(\iota_1(a))\quad \quad\quad\text{by~} \bet (\iota_2\circ q)\circ\iota_1=\iota_2\circ q
\\&=\Psi(\pi(\iota_1(a))) \quad \quad\quad \quad\text{by~} \Psi\circ \pi=\bet (\iota_2\circ q)
\\&=\Psi(\iota(q(a))) \quad\quad\quad \quad\quad \quad\text{by~} \pi\circ \iota_1=\iota\circ q
\\&=\Psi(\iota(a/I)) \quad \quad\quad \quad\quad\text{by~} q(a)=a/I
\end{align*}


Then
\[\Psi\circ \bet\iota :\bet(A/I)\rightarrow \bet(A/I)\]
 is a continuous \(MV\)-homomorphism. Now restrict \(\Psi\circ \bet\iota\) to the dense subalgebra \(\iota_2(A/I)\) of \(\bet(A/I)\). For any \(a/I \in A/I\) we have
\begin{align*}
(\Psi\circ \bet\iota)(\iota_2(a/I))) &= \Psi( \bet\iota(\iota_2(a/I))
\\&=\Psi(\iota(a/I)) \quad \quad\quad\text{by~} \bet\iota\circ \iota_2=\iota
\\&=\iota_2(a/I) \quad \quad\quad \quad\text{by~} \Psi\circ \iota=\iota_2.
\end{align*}

This  implies that restricting \(\Psi \circ \beta\iota\) on \(\iota_2(A/I)\) is identity. By continuity and density of  \(\iota_2(A/I)\) in \(\beta(A/I)\) we obtain \(\Psi \circ \beta\iota\) is the identity mapping on \(\beta(A/I)\), which implies that \(\Psi\) is a continuous bijection on \( \bet A / \cl_{\bet A},  \iota_1(I)\), because \(\beta\iota\) is surjective.
Hence \(\Psi\) is an isomorphism of compact topological \(MV\)-algebras.  In particular, its kernel is trivial:
\[
\ker \Psi = K_q / \cl_{\bet A}\iota_1(I) = 0.
\]
Therefore \(K_q = \cl_{\bet A} \iota_1(I)\).

Since \(K_q = \cl_{\bet A} \iota_1(I)\), we have \(\bet A / K_q = \bet A / \cl_{\bet A} \iota_1(I)\). The map \(\Psi : \bet A/\cl_{\bet A} \iota_1(I) \to \bet(A/I)\) is already known to be a continuous bijection between compact Hausdorff spaces, hence a homeomorphism and an isomorphism of topological \(MV\)-algebras.  Thus
\[
\bet A / \cl_{\bet A} \iota_1(I) \cong \bet(A/I)
\]
holds unconditionally.
\end{proof}

\begin{corollary}\label{cor:compact-ideal-exactness}
Let \(I\) be a compact ideal of a pseudocompact topological \(MV\)-algebra \(A\). Then
\(\bet A /I \cong \bet(A/I)\) holds.
\end{corollary}

\begin{proof}
Since every compact Hausdorff subset is closed, \(I\) is closed in \(A\) and \(\cl_{\bet A} \iota_1(I)=I\). Hence, the fact \(\bet A /I \cong \bet(A/I)\) follows from Theorem \ref{thm:quotient-exactness}.
\end{proof}

The compact topological \(MV\)-algebras that occur as Stone--\v Cech compactifications can be described exactly in terms of dense \(C^*\)-embedded subalgebras.

\begin{definition}
Let \(X\) be a dense subspace of a compact Hausdorff space \(K\).  We say that \(X\) is \(C^*\)-embedded in \(K\) if every bounded continuous real-valued function on \(X\) extends continuously to \(K\).
\end{definition}

\begin{theorem}\label{thm:compact-representation-final}
Let \(B\) be a compact Hausdorff topological \(MV\)-algebra and let \(A\subseteq B\) be a dense \(MV\)-subalgebra with the subspace topology.  Then the following are equivalent:
\begin{enumerate}[label=\textup{(\arabic*)}]
\item[(a)] The canonical continuous extension \( p:\beta A\longrightarrow B\)
of the inclusion \(A\hookrightarrow B\) is a homeomorphism;
\item[(b)] \(A\) is \(C^*\)-embedded in \(B\);
\item[(c)] \(B\) is a Stone--\v Cech compactification of \(A\).
\end{enumerate}
When these conditions hold, \(p\) is a topological \(MV\)-isomorphism.  Conversely, every compact Hausdorff topological \(MV\)-algebra arises in this way, trivially by taking \(A=B\).
\end{theorem}

\begin{proof}
The equivalence between \(C^*\)-embeddedness of a dense subspace \(A\subseteq B\) and the statement that \(B\) is the Stone--\v Cech compactification of \(A\) is the standard characterization of \(\beta A\).  Thus (a), (b), and (c) are equivalent as statements about compactifications.

Assume these conditions hold.  The map \(p:\beta A\to B\) is the unique continuous extension of the inclusion \(A\hookrightarrow B\).  It remains only to check that it respects the \(MV\)-operations.  Consider the two continuous maps
\[
        \beta A\times\beta A\longrightarrow B
\]
given by
\[
        (u,v)\longmapsto p(u\oplus v)
        \quad\text{and}\quad
        (u,v)\longmapsto p(u)\oplus p(v).
\]
They agree on the dense subset \(A\times A\), because \(A\subseteq B\) is an \(MV\)-subalgebra.  Since \(B\) is Hausdorff, they agree everywhere.  Hence \(p\) preserves \(\oplus\).  The same argument applied to the two maps \(u\mapsto p(u^*)\) and \(u\mapsto p(u)^*\) shows that \(p\) preserves the involution.  Thus \(p\) is a topological \(MV\)-isomorphism.

The converse is immediate: if \(B\) is compact Hausdorff, then \(\beta B=B\), so taking \(A=B\) gives the stated representation.
\end{proof}

\begin{corollary}
Let \(B\) be a compact Hausdorff topological \(MV\)-algebra.  A dense \(MV\)-subalgebra \(A\subseteq B\) represents \(B\) as \(\beta A\) if and only if \(A\) is \(C^*\)-embedded in \(B\).
\end{corollary}

\section{Concluding remarks}

Topological \(MV\)-algebras do not have invertible translations, so the topological group proof is unavailable.  Nevertheless they have a continuous Mal'tsev operation, and the Reznichenko--Uspenskij theorem shows that this is exactly the structure needed for the pseudocompact product theorem.  Thus the correct conceptual diagram is
\[
        \text{topological groups}
        \subseteq
        \text{Mal'tsev spaces}
        \supseteq
        \text{topological }MV\text{-algebras}.
\]

We have proved that arbitrary products of pseudocompact topological \(MV\)-algebras are pseudocompact.  The proof does not use a Ra\v{\i}kov completion or a group-like uniformity.  Instead, it relies on the fact that every topological \(MV\)-algebra is a Mal'tsev space and on the theorem of Reznichenko and Uspenskij extending the Comfort--Ross theorem from topological groups to Mal'tsev spaces. The Stone-\v{C}ech compactification theorem obtained here suggests further questions.  For instance, one may ask which algebraic or topological properties of \(A\) are reflected in the compact topological \(MV\)-algebra \(\beta A\), how ideals of \(A\) behave under the compactification, and whether spectral or maximal-ideal constructions commute with this compactification under additional hypotheses.

\end{document}